\documentclass[11pt,oneside,english]{amsart}
\usepackage[utf8]{inputenc}
\usepackage{geometry}
\usepackage{color}
\usepackage{babel}
\usepackage{amstext}
\usepackage{amsthm}
\usepackage{amssymb}
\usepackage{setspace}
\PassOptionsToPackage{normalem}{ulem}
\usepackage{ulem}
\usepackage[unicode=true,pdfusetitle,
 bookmarks=true,bookmarksnumbered=false,bookmarksopen=false,
 breaklinks=false,pdfborder={0 0 1},backref=page,colorlinks=true]
 {hyperref}
\hypersetup{
 linkcolor=blue}

\makeatletter
\numberwithin{equation}{section}
\numberwithin{figure}{section}
\theoremstyle{plain}
\newtheorem{thm}{\protect\theoremname}[section]
\theoremstyle{definition}
\newtheorem{defn}[thm]{\protect\definitionname}
\theoremstyle{definition}
\newtheorem{example}[thm]{\protect\examplename}
\theoremstyle{plain}
\newtheorem{lem}[thm]{\protect\lemmaname}
\theoremstyle{plain}
\newtheorem{cor}[thm]{\protect\corollaryname}
\theoremstyle{remark}
\newtheorem{rem}[thm]{\protect\remarkname}
\theoremstyle{plain}
\newtheorem{conjecture}[thm]{\protect\conjecturename}
\theoremstyle{plain}
\newtheorem{question}[thm]{\protect\questionname}
\theoremstyle{plain}
\newtheorem{prop}[thm]{\protect\propositionname}

\usepackage{graphicx}
\usepackage{appendix}

\usepackage{enumitem}
\usepackage[backref=page]{hyperref}
\usepackage[all]{xy}
\usepackage{multicol}

\makeatother

\providecommand{\conjecturename}{Conjecture}
\providecommand{\corollaryname}{Corollary}
\providecommand{\definitionname}{Definition}
\providecommand{\examplename}{Example}
\providecommand{\lemmaname}{Lemma}
\providecommand{\propositionname}{Proposition}
\providecommand{\questionname}{Question}
\providecommand{\remarkname}{Remark}
\providecommand{\theoremname}{Theorem}

\begin{document}
\global\long\def\F{\mathrm{\mathbf{F}} }%
\global\long\def\Aut{\mathrm{Aut}}%
\global\long\def\C{\mathbf{C}}%
\global\long\def\H{\mathcal{H}}%
\global\long\def\U{\mathcal{U}}%
\global\long\def\P{\mathcal{P}}%
\global\long\def\ext{\mathrm{ext}}%
\global\long\def\hull{\mathrm{hull}}%
\global\long\def\triv{\mathrm{triv}}%
\global\long\def\Hom{\mathrm{Hom}}%
\global\long\def\trace{\mathrm{tr}}%
\global\long\def\End{\mathrm{End}}%
\global\long\def\Z{\mathbf{Z}}%
\global\long\def\GL{\mathrm{GL}}%
\global\long\def\T{\mathbb{T}}%
\global\long\def\df{\mathrm{def}}%
\global\long\def\Epi{\mathrm{Epi}}%
\global\long\def\R{\mathbf{R}}%
\global\long\def\FS{\mathrm{FS}}%
\global\long\def\O{\mathrm{O}}%
 
\global\long\def\Sp{\mathrm{Sp}}%
 
\global\long\def\U{\mathrm{U}}%
\global\long\def\cl{\mathrm{cl}}%
\global\long\def\scl{\mathrm{scl}}%
\global\long\def\Q{\mathbf{Q}}%
\global\long\def\N{\mathbf{N}}%
\global\long\def\Sym{\mathrm{Sym}}%
\global\long\def\std{\mathrm{std}}%

\title{Automorphism-invariant positive definite functions on free groups }
\author{Benoît Collins, Michael Magee, Doron Puder}
\begin{abstract}
In this article we raise some new questions about positive definite
functions on free groups, and explain how these are related to more
well-known questions. The article is intended as a survey of known
results that also offers some new perspectives and interesting observations;
therefore the style is expository.

\tableofcontents{}
\end{abstract}

\maketitle

\section{Introduction\label{sec:Introduction}}

Fix $r\geq1$ and write $\F=\F_{r}$ for a free group on $r$ generators
$\{x_{1},\ldots,x_{r}\}$. A central role in this paper will be played
by the automorphism group $\Aut(\F)$ of $\F$. It was proved by Nielsen
\cite{Nielsen} that $\Aut(\F)$ is generated by the following \emph{elementary
Nielsen moves:}
\begin{itemize}
\item For $\sigma$ an element of the symmetric group $S_{r}$, we have
$\alpha_{\sigma}\in\Aut(\F)$ where
\begin{equation}
\alpha_{\sigma}(x_{1},\ldots,x_{r})\stackrel{\mathrm{def}}{=}(x_{\sigma(1)},\ldots,x_{\sigma(r)}).\label{eq:nielsen-perm}
\end{equation}
\item We have $\iota\in\Aut(\F)$ where 
\begin{equation}
\iota(x_{1},x_{2},\ldots,x_{r})=(x_{1}^{-1},x_{2},\ldots,x_{r}).\label{eq:nielsen-inverse}
\end{equation}
\item We have $\gamma\in\Aut(\F)$ where
\begin{equation}
\gamma(x_{1},x_{2},\ldots,x_{r})=(x_{1}x_{2},x_{2},\ldots,x_{r}).\label{eq:nielsen-mult}
\end{equation}
\end{itemize}
The other central concept of this paper is a \emph{positive definite
function }on a group. 
\begin{defn}
\label{def:positive-definite}Let $\Gamma$ be any discrete group.
A function $\tau:\Gamma\to\C$ is called \emph{positive definite}
if for any finite subset $S\subset\Gamma$, the matrix 
\[
[\tau(\gamma'\gamma^{-1})]_{\gamma,\gamma'\in S}
\]
is positive semi-definite. In other words, for any vector $(\alpha_{\gamma})_{\gamma\in S}\in\C^{S}$
we have
\[
\sum_{\gamma,\gamma'\in S}\tau(\gamma'\gamma^{-1})\alpha_{\gamma'}\overline{\alpha_{\gamma}}\geq0.
\]
\end{defn}

If $\Gamma$ is a discrete group, the group $\Aut(\Gamma)$ acts by
precomposition on the collection of positive definite functions on
$\Gamma$, giving rise to the notion of $\Aut(\Gamma)$-invariant
positive definite functions. Explicitly, a positive definite function
$\tau$ is $\Aut(\Gamma)$-invariant if 
\[
\tau(\alpha(\gamma))=\tau(\gamma),\quad\forall\gamma\in\Gamma,\,\forall\alpha\in\Aut(\Gamma).
\]
In this paper we are mainly interested in the case $\Gamma=\F$. Positive
definite functions on free groups, without the $\Aut(\F)$-invariance
condition, have been the subject of various investigations \cite{DF,Bozejko,BT},
stemming in part from a fundamental construction of Haagerup in \cite{Haagerup}.
See also the monograph \cite{FP}.

\emph{Our aim here is to explain what is known about $\Aut(\F)$-invariant
positive definite functions on $\F$, and identify some important
questions about them.}

\begin{example}
\label{exa:reg-defn}Let $\tau_{\lambda}(e)=1$ and $\tau_{\lambda}(w)=0$
for $w\neq e$. One can directly verify that this is a positive definite
function on $\F$, and that $\tau_{\lambda}$ is $\Aut(\F)$-invariant.
\end{example}

\begin{example}
\label{exa:triv-definition}Let $\tau_{\triv}(w)=1$ for all $w\in\F$.
This is another $\Aut(\F)$-invariant positive definite function on
$\F$. 
\end{example}

A rich family of examples that are the subject of much ongoing work
arise from \emph{word maps. }Throughout the rest of this paper, $G$
will always refer to a compact topological group, and $\mu$ will
be its probability Haar measure. \emph{In this paper, all topological
groups are assumed to be Hausdorff}\footnote{It is convenient to assume this so that we can identify Borel measures
on $G$ or $G^{r}$ with elements of the continuous linear dual of
continuous functions, without getting into technicalities.}\emph{.} Denote $G^{r}\stackrel{\mathrm{def}}{=}\underbrace{G\times G\times\ldots\times G}_{r~\mathrm{times}}$.
Any $w\in\F_{r}$ gives rise to a \emph{word map}
\[
w:G^{r}\to G
\]
defined by substitutions. For example, if $r=2$ and $w=x_{1}^{2}x_{2}^{-2}$,
then $w(g_{1},g_{2})=g_{1}^{2}g_{2}^{-2}$. A related concept is that
of the $w$-measure on $G$. The $w$-measure is the law of the random
variable obtained by picking $r$ independent elements of $G$ according
to the Haar measure, and evaluating the word map $w$ at this random
tuple. More formally, the \emph{$w$-measure on $G$ }is the pushforward
measure
\[
\mu_{w}=w_{*}(\mu^{r}),
\]
where $\mu^{r}$ is the Haar measure on $G^{r}$. Word maps and measures
give rise to $\Aut(\F)$-invariant positive definite functions on
$\F$ as follows:
\begin{example}[Compact group construction]
\label{exa:compact-group-construction}Let $G$ be a compact topological
group and $\pi:G\to U(V)$ be an unitary representation of $G$, with
$V$ a finite dimensional vector space over $\C$. We define
\[
\tau_{G,\pi}:\F\to\C
\]
by
\[
\tau_{G,\pi}(w)\stackrel{\df}{=}\int_{g\in G}\trace(\pi(g))d\mu_{w}=\int_{{\bf g}\in G^{r}}\trace(\pi(w({\bf g})))d\mu^{r}({\bf g}).
\]
In other words, this function maps $w\in\F$ to the expected value
of the character of $\pi$ under the $w$-measure $\mu_{w}$. This
is a positive definite function on $w$ as follows. Suppose $S\subset\F$
and we are given $\alpha_{w}\in\C$ for each $w\in S$. Then
\begin{align}
\sum_{w,w'\in S}\tau_{G,\pi}(w'w^{-1})\alpha_{w'}\overline{\alpha_{w}} & =\sum_{w,w'\in S}\alpha_{w'}\overline{\alpha_{w}}\int_{{\bf g}\in G^{r}}\trace(\pi([w'w^{-1}]({\bf g})))d\mu^{r}({\bf g})\nonumber \\
 & =\int_{{\bf g}\in G^{r}}\trace(A_{{\bf {\bf g}}}A_{{\bf g}}^{*})d\mu^{r}({\bf g}),\label{eq:positivity}
\end{align}
where a superscript $*$ means conjugate transpose and 
\[
A_{{\bf g}}=\sum_{w\in S}\alpha_{w}\pi(w({\bf g}))\in\End(V).
\]
Hence the quantity (\ref{eq:positivity}) is an integral of traces
of non-negative operators and hence must be non-negative.

Moreover, $\tau_{G,\pi}$ is $\Aut(\F)$-invariant. This will follow
from the following lemma that is folklore\footnote{See for example \cite{GOLDMAN} where a version is stated without
a proof in the second sentence, and the unpublished paper \cite[Section 2.5]{MP}.}. In this paper, all proofs are given in the Appendix, and we mark
all statements with proofs in the Appendix by a $\star$.
\end{example}

\begin{lem}
\label{lem:Haar-measure-Aut-invariant}$\star$ The action of $\Aut(\F)$
by precomposition on $\Hom(\F,G)\cong G^{r}$ preserves the Haar measure
$\mu^{r}$.
\end{lem}

The following two corollaries are immediate consequences of Lemma
\ref{lem:Haar-measure-Aut-invariant}.
\begin{cor}
\label{cor:w-measure-aut-invariant}If $G$ is a compact topological
group, $w\in\F$, and $\alpha\in\Aut(\F)$, then the $w$-measure
$\mu_{w}$ on $G$ is equal to the $\alpha(w)$-measure $\mu_{\alpha(w)}$
on $G$, namely,
\[
\mu_{w}=\mu_{\alpha(w)}.
\]
\end{cor}

\begin{cor}
If $G$ is a compact topological group and $\pi$ is a finite dimensional
unitary representation of $G$, the positive definite function $\tau_{G,\pi}$
on $\F$ given in Example \ref{exa:compact-group-construction} is
$\Aut(\F)$-invariant.
\end{cor}

Note that this family of positive definite functions on $\F$ coming
from compact groups, includes, in particular, those coming from \emph{finite}
groups.

\begin{rem}
We point out that the construction given in Example \ref{exa:compact-group-construction}
also works if $\mu^{r}$ is replaced by any $\Aut(\F)$-invariant
Borel measure on $G^{r}$. These measures are by no means classified,
and we will return to this point later in Question \ref{que:classification-of-invariant-borel-measures}. 
\end{rem}

In light of Corollary \ref{cor:w-measure-aut-invariant}, the following
conjectures have been put forward:
\begin{conjecture}
\label{que:compact-groups-separate-orbits}Suppose $w_{1},w_{2}$
are in $\F$. If the $w$-measures $\mu_{w_{1}}$ and $\mu_{w_{2}}$
are the same on all compact groups $G$, does it follow that $w_{2}\in\Aut(\F).w_{1}$?
\end{conjecture}

It has even been conjectured that:
\begin{conjecture}[Shalev]
\label{que:finite-groups-separate-orbits}If the $w$-measures $\mu_{w_{1}}$
and $\mu_{w_{2}}$ are the same on all \uline{finite} groups $G$,
then $w_{2}\in\Aut(\F).w_{1}$.
\end{conjecture}

See \cite[Question 2.2]{AV} where Conjecture \ref{que:finite-groups-separate-orbits}
was posed as a question; the conjecture was made by Shalev in \cite[Conj. 4.2]{Shalev2013}.
Of course Conjecture \ref{que:compact-groups-separate-orbits} is
a direct consequence of Conjecture \ref{que:finite-groups-separate-orbits}.
In this paper we introduce the following related (weaker) question:
\begin{question}
\label{que:Do--invariant-positive-functions-sep-orbits}Do $\Aut(\F)$-invariant
positive definite functions on $\F$ separate $\Aut(\F)$-orbits?
In other words, if $w_{1},w_{2}$ are in $\F$ and $\tau(w_{1})=\tau(w_{2})$
for all $\Aut(\F)$-invariant positive definite functions $\tau$
on $\F$, does it follow that $w_{2}\in\Aut(\F).w_{1}$?
\end{question}

An affirmative answer to Question \ref{que:Do--invariant-positive-functions-sep-orbits}
could be viewed as an orbital analog of the Gelfand-Raikov Theorem
\cite{GR}: for any locally compact topological group $\mathcal{G}$,
the positive definite functions on $\mathcal{G}$ separate elements
of $\mathcal{G}$. Indeed, Question \ref{que:Do--invariant-positive-functions-sep-orbits}
could be asked for any locally compact topological group, but we restrict
our attention here to the important special case of free groups. 

To compare Question \ref{que:Do--invariant-positive-functions-sep-orbits}
and Conjectures \ref{que:compact-groups-separate-orbits} and \ref{que:finite-groups-separate-orbits},
we introduce some equivalence relations on $\F$. For $w_{1},w_{2}\in\F$
we say
\begin{itemize}
\item $w_{1}\stackrel{\Aut(\F)}{\sim}w_{2}$ if $w_{2}\in\Aut(\F).w_{1}$
\item $w_{1}\stackrel{\text{\textbf{FinGrp}}}{\sim}w_{2}$ if $\mu_{w_{1}}=\mu_{w_{2}}$
on any finite group.
\item $w_{1}\stackrel{\mathrm{\textbf{CptGrp}}}{\sim}w_{2}$ if the measures
$\mu_{w_{1}}=\mu_{w_{2}}$ on any compact group.
\item $w_{1}\stackrel{\textbf{PosDef}}{\sim}w_{2}$ if $\tau(w_{1})=\tau(w_{2})$
for all $\Aut(\F)$-invariant positive definite functions $\tau$
on $\F$.
\end{itemize}
For $w_{1},w_{2}\in\F$, we have 
\begin{equation}
w_{1}\stackrel{\Aut(\F)}{\sim}w_{2}\stackrel{\mathrm{}}{\implies}w_{1}\stackrel{\textbf{PosDef}}{\sim}w_{2}\stackrel{\mathrm{}}{\implies}w_{1}\stackrel{\mathrm{\textbf{CptGrp}}}{\sim}w_{2}\stackrel{\mathrm{}}{\implies}w_{1}\stackrel{\text{\textbf{FinGrp}}}{\sim}w_{2}.\label{eq:relations-impliications}
\end{equation}
The first and last implications above are obvious. The second implication
follows immediately from the following lemma.
\begin{lem}
\label{lem:W-MEASURE-determined-by-character-integrals}$\star$ For
any compact topological group $G$, and $w\in\F$, the $w$-measure
$\mu_{w}$ on $G$ is determined uniquely by the map $\pi\mapsto$$\tau_{G,\pi}(w)$
where $\pi$ runs over irreducible unitary representations of $G$
and $\tau_{G,\pi}$ are the functions constructed in Example \ref{exa:compact-group-construction}.
\end{lem}

\begin{rem}
Section 8 in \cite{PP15} discusses a few other related equivalence
relations between words, where the focus is on word measures on finite
groups and the profinite topology on the free group.
\end{rem}

\addtocontents{toc}{\protect\setcounter{tocdepth}{1}}

\subsection*{Notation}

We write $e$ for the identity element of a group. If $A$ and $B$
are elements of the same group, then $[A,B]=ABA^{-1}B^{-1}$ is their
commutator. If $H$ is a group, then $[H,H]$ denotes its commutator
subgroup. We write $\emptyset$ for the empty set.

\subsection*{Acknowledgments}

BC was supported by JSPS KAKENHI 17K18734 and 17H04823. DP was supported
by the Israel Science Foundation (grant No. 1071/16).

\addtocontents{toc}{\protect\setcounter{tocdepth}{2}}

\section{A survey}

In this section we give a brief survey describing current knowledge
about word measures on groups and Conjectures \ref{que:compact-groups-separate-orbits}
and \ref{que:finite-groups-separate-orbits}.

It has been known for a while that several properties of free words
can be detected in finite quotients of free words and therefore also
in word measures on finite groups. For example, if a word $w\in\F$
is \emph{not} an $n$-th power (namely, if there is no $u\in\F$ with
$w=u^{n}$) then there is a normal subgroup $N\trianglelefteq\F$
such that $wN$ is not an $n$-th power in $Q=\F/N$ -- this result
is attributed to Lubotzky in \cite{thompson1997power}. It follows
that if $w_{1}$ is an $n$-th power and $w_{2}$ is not, then for
some finite group $Q$, there is an element $q\in Q$ which is not
an $n$-th power, such that $q\in w_{2}\left(Q^{r}\right)$ but $q\notin w_{1}\left(Q^{r}\right)$.
Thus $w_{1}\stackrel{\text{\textbf{FinGrp}}}{\not\sim}w_{2}$. Consult
\cite{Hanany} for a different argument yielding this last result.

Similarly, Khelif \cite{khelif2004finite} shows that if $w\in\F$
is \emph{not} a commutator of two words, then its image in some finite
quotient of $\F$ is a non-commutator. It follows that if $w_{1}$
is a commutator and $w_{2}$ not, then $w_{1}\stackrel{\text{\textbf{FinGrp}}}{\not\sim}w_{2}$.

However, the first significant progress on Conjecture \ref{que:finite-groups-separate-orbits}
came from an important special case. If $w\in\Aut(\F).x_{1}$, then
$w$ is called \emph{primitive. }A word is primitive in $\F_{r}$
if and only if it is a member of a generating set of $\F_{r}$ of
size $r$. Since it is clear that $\mu_{x_{1}}=\mu$, i.e.~Haar measure
on $G$, for any compact $G$, it follows from Corollary \ref{cor:w-measure-aut-invariant}
that if $w$ is a primitive word then $\mu_{w}=\mu$ on any compact
$G$, and in particular, on any finite $G$. The following theorem,
asserting the converse, was conjectured to hold independently by several
people including Avni, Gelander, Larsen, Lubotzky, and Shalev:
\begin{thm}[Puder-Parzanchevski]
\label{conj:measure-preserving-words-primitive}If $w\in\F_{r}$
and $\mu_{w}=\mu$ on every finite group, then $w$ is primitive.
\end{thm}

Theorem \ref{conj:measure-preserving-words-primitive} was first proved
by Puder \cite[Theorem 1.5]{Puder2014} when $r=2$ and then proved
for general $r\geq3$ by Puder and Parzanchevski in \cite[Theorem 1.4']{PP15}.
Both papers rely on a careful analysis of the functions $\tau_{G,\pi}$
constructed in Example \ref{exa:compact-group-construction} when
$G=S_{n}$, the symmetric group on $n$ letters, and $\pi$ is the
standard $n$-dimensional representation of $S_{n}$ by $0$-$1$
matrices. Theorem \ref{conj:measure-preserving-words-primitive} can
be restated as the implication
\[
w\stackrel{\text{\textbf{FinGrp}}}{\sim}x_{1}\implies w\stackrel{\Aut(\F)}{\sim}x_{1}
\]
and therefore establishes a basic instance of Conjecture \ref{que:finite-groups-separate-orbits}.

The word $x_{1}$, and by extension, the primitive words, have the
property that whenever $\pi$ is an \emph{irreducible} representation
of the compact group $G$, the values $\tau_{G,\pi}(x_{1})$ of Example
\ref{exa:compact-group-construction} are given by a very simple formula.
Indeed, suppose that $\pi$ is an irreducible unitary representation.
Then
\begin{equation}
\tau_{G,\pi}(x_{1})=\begin{cases}
1 & \text{if \ensuremath{\pi} is the trivial representation}\\
0 & \text{otherwise.}
\end{cases}\label{eq:schur-orth}
\end{equation}
This is a direct consequence of Schur orthogonality. 

There is another type of words with a similarly general exact expression
for $\tau_{G,\pi}(w)$, namely, \emph{surface words. }An \emph{orientable
surface word} is one of the form
\[
s_{g}=[x_{1},x_{2}]\cdots[x_{2g-1},x_{2g}]
\]
where we assume $g\geq1$ and $r\geq2g$ (recall that $r$ is the
rank of the free group $\F=\F_{r}$). A \emph{non-orientable surface
word} is one of the form
\[
t_{g}=x_{1}^{2}\cdots x_{g}^{2}
\]
where $g\geq1$ and $r\geq g$. The reason for this naming is that
the one-relator groups
\[
\Gamma_{g}=\left\langle \F_{2g}\,|\,s_{g}\right\rangle ,\quad\Lambda_{g}=\left\langle \F_{g}\,|\,t_{g}\right\rangle 
\]
are respectively, the fundamental groups of a closed \uline{orientable}
surface of genus $g$, and a closed \uline{non-orientable} surface
of genus $g$ (the connected sum of $g$ copies of the real projective
plane $\mathbb{P}^{2}(\R)$).

Frobenius \cite{frobenius1896gruppencharaktere} proved the following
result for finite groups, but the same proof applies to compact groups
in general.
\begin{thm}
\label{thm:frob}Suppose that $\pi$ is an irreducible representation
of the compact group $G$ on the vector space $V$. Then
\[
\tau_{G,\pi}\left([x_{1},x_{2}]\right)=\frac{1}{\dim V}.
\]
\end{thm}

An analogous result was later proved by Frobenius and Schur \cite{frobenius1906reellen}.
\begin{thm}
\label{thm:frob-schur}Suppose that $\pi$ is an irreducible unitary
representation of the compact group $G$ on the vector space $V$.
Then $\tau_{G,\pi}(x_{1}^{2})$ is in $\{-1,0,1\}$ and is called
the Frobenius-Schur indicator of $\pi$, denoted by $\FS(\pi)$. The
Frobenius-Schur indicator is also given by
\[
\FS(\pi)=\begin{cases}
1 & \text{if \ensuremath{\pi} is equivalent to a real representation}\\
0 & \text{\text{if \ensuremath{\trace(\pi)} is not real}}\\
-1 & \text{if \ensuremath{\trace(\pi)} is real, but \ensuremath{\pi} is not equivalent to a real representation.}
\end{cases}
\]
\end{thm}

One also has the following basic observation.
\begin{lem}
\label{lem:convolution-of-words}$\star$ If $w_{1},w_{2}\in\F$,
and $w_{1}$ and $w_{2}$ are generated by disjoint sets of the $x_{i}$,
then 
\[
\mu_{w}=\mu_{w_{1}}*\mu_{w_{2}}.
\]
where $*$ denotes convolution. Hence for $\pi$ irreducible
\[
\tau_{G,\pi}(w)=\frac{1}{\dim V}\tau_{G,\pi}(w_{1})\tau_{G,\pi}(w_{2}).
\]
\end{lem}

From Lemma \ref{lem:convolution-of-words} and Theorems \ref{thm:frob}
and \ref{thm:frob-schur} it immediately follows that for irreducible
$\pi$
\begin{equation}
\tau_{G,\pi}(s_{g})=\frac{1}{(\dim V)^{2g-1}},\label{eq:orientable-value}
\end{equation}
and 
\begin{equation}
\tau_{G,\pi}(t_{g})=\frac{\FS(\pi)^{g}}{(\dim V)^{g-1}}.\label{eq:nonorientable-value}
\end{equation}
By Lemma \ref{lem:W-MEASURE-determined-by-character-integrals}, this
fully describes the word measures $\mu_{s_{g}}$ and $\mu_{t_{g}}$
on all compact groups. The following theorem was suggested as a line
of inquiry at the 27th International Conference in Operator Theory
in Timişoara, and has since been established to hold \cite{MP3}.
\begin{thm}[Magee-Puder]
\label{thm:surface-words-sep}If $w\in\F_{r}$ and $\mu_{w}=\mu_{s_{g}}$
on every compact group, then ($r\ge2g$, and) $w\in\Aut(\F_{r}).s_{g}$.
In other words,
\[
w\stackrel{\mathbf{CptGrp}}{\sim}s_{g}\implies w\stackrel{\Aut(\F_{r})}{\sim}s_{g}.
\]
If $w\in\F_{r}$ and $\mu_{w}=\mu_{t_{g}}$ on every compact group,
then ($r\ge g$, and) $w\in\Aut(\F_{r}).t_{g}$. In other words,
\[
w\stackrel{\mathbf{CptGrp}}{\sim}t_{g}\implies w\stackrel{\Aut(\F_{r})}{\sim}t_{g}.
\]
\end{thm}

One may view this as a converse to the results of Frobenius and Schur:
the formulas (\ref{eq:orientable-value}) and (\ref{eq:nonorientable-value})
uniquely characterize the orbits of $s_{g}$ and $t_{g}$. The proof
of Theorem \ref{thm:surface-words-sep} involves an analysis of the
values $\tau_{G,\pi}(w)$ where $G,\pi$ are one of the following:
\begin{itemize}
\item $G=\U(n)$, the group of $n\times n$ complex unitary matrices, and
$\pi$ is the $n$-dimensional defining representation of $\U(n)$.
This relies on the results of the paper \cite{MP2}.
\item $G=\text{\ensuremath{\O}}(n)$, the group of $n\times n$ real orthogonal
matrices, and $\pi$ is the $n$-dimensional defining representation
of $\text{\ensuremath{\O}}(n)$. The necessary analysis here comes
from the work \cite{MP4}.
\item $G=S_{n,m}$ or $G=S^{1}\wr S_{n}$ a \emph{generalized symmetric
group, namely, }the group of all $n\times n$ complex matrices such
that any row or column contains exactly one non-zero entry, and the
non-zero entries are taken from the $m$th roots of $1$ or from the
entire unit circle $S^{1}$. The representation $\pi$ is the standard
one given by the definition of the group as a matrix group. The necessary
analysis here is developed in \cite{MP3}.
\end{itemize}
Independently, Hanany, Meiri and Puder obtained the following result
\cite{Hanany}:
\begin{thm}[Hanany-Meiri-Puder]
\label{thm:commutator-and-primitive-powers-separated-finite} Let
$w_{0}=x_{1}^{~m}$ or $w_{0}=\left[x_{1},x_{2}\right]^{m}$ for some
$m\in\N$. If $w\in\F_{r}$ induces the same measure as $w_{0}$ on
every \uline{finite} group, then $w\in\Aut(\F_{r}).w_{0}$. In
other words,
\begin{equation}
w\stackrel{\mathbf{FinGrp}}{\sim}w_{0}\implies w\stackrel{\Aut(\F_{r})}{\sim}w_{0}.\label{eq:w0}
\end{equation}
\end{thm}

In particular, Theorem \ref{thm:commutator-and-primitive-powers-separated-finite}
strengthens Theorem \ref{thm:surface-words-sep} in the case $w_{0}=\left[x_{1},x_{2}\right]$.
It is an interesting question whether Theorem \ref{thm:surface-words-sep}
can be proved for general $g$ using only finite groups $G$. The
proof of Theorem \ref{thm:commutator-and-primitive-powers-separated-finite}
relies on the results of Lubotzky \cite{thompson1997power} and Khelif
\cite{khelif2004finite} mentioned above, as well as on further developing
the analysis of word measures on $S_{n}$ from \cite{PP15}. In fact,
it is shown in \cite{Hanany} that whenever (\ref{eq:w0}) holds for
some word $w_{0}\in\F$, it also holds for every power of $w_{0}$.

\subsubsection*{Rational Functions}

A recurring theme in many of the works mentioned above is that for
many ``natural'' families of groups and representations $\left\{ \left(G_{n},\pi_{n}\right)\right\} _{n\ge N_{0}}$,
the function $\tau_{G_{n},\pi_{n}}\left(w\right)$ is given by a rational
function in $n$. For example, if $\mathrm{std}$ is the defining
$n$-dimensional representation of $\U\left(n\right)$, then for $n\ge2$
\[
\tau_{U\left(n\right),\mathrm{std}}\left(\left[x_{1},x_{2}\right]^{2}\right)=\frac{-4}{n^{3}-n}.
\]
Indeed, this phenomenon occurs for natural series of representations
of $S_{n}$ \cite{nica1994number,Linial2010} and for the defining
representations of generalized symmetric groups \cite{MP3}. Using
the Weingarten calculus developed for computing integrals over Haar-random
elements of classical compact Lie groups \cite{weingarten1978asymptotic,collins2003moments,CS},
it is shown to hold also in the case of natural families of representations
of $\U\left(n\right)$ \cite{Radulescu06,MSS07} and of $\O\left(n\right)$
and $\Sp\left(n\right)$ \cite{MP4}. The same phenomenon also occurs
for natural families of representations of $\mathrm{GL}_{n}\left(\mathbb{F}_{q}\right)$,
where $\mathbb{F}_{q}$ is a fixed finite field \cite{West}. 

These rational expressions depend on $w$, of course, but are $\mathrm{Aut}(\F)$-invariant.
This means that they should have an ``$\mathrm{Aut(}\F)$-invariant''
interpretation, not relying on combinatorial properties of $w$, but
rather on properties of $w$ as an element of the abstract free group
(with no given basis). Finding such interpretation for at least some
of terms of the rational functions is one of the main results of \cite{PP15}
in the case of $S_{n}$, of \cite{MP2} in the case of $\U\left(n\right)$,
of \cite{MP4} in the cases of $\O\left(n\right)$ and $\Sp\left(n\right)$
and of \cite{MP3} in the case of generalized symmetric groups.\medskip{}

One plausible strategy for proving Conjectures \ref{que:compact-groups-separate-orbits}
and \ref{que:finite-groups-separate-orbits} is to gather a list of
invariants of words which can be determined by word measures on groups,
and then prove that this list separates $\mathrm{Aut}(\F)$-orbits.
We have already mentioned above two invariants that can be determined
by word measures on finite groups: whether $w$ is an $n$-th power,
and whether $w$ is a simple commutator. Now we turn to a result of
a similar type, but with a richer invariant that is detected. Given
$w\in\F$, we define the \emph{commutator length }of $w$, denoted
$\cl(w)$, to be the minimum $g$ for which we can solve the equation
\[
w=[u_{1},v_{1}]\cdots[u_{g},v_{g}]
\]
for $u_{i},v_{i}\in\F$. If it is not possible to write $w$ as the
product of commutators (i.e., if $w\notin[\F,\F]$) then we say $\cl(w)=\infty$.
There is a related concept of stable commutator length. The stable
commutator length of $w$, denoted $\scl(w)$, is defined by 
\[
\scl(w)\stackrel{\df}{=}\lim_{m\to\infty}\frac{\cl(w^{m})}{m},\quad\text{if \ensuremath{w\in[\F,\F]},}
\]
or $\infty$ otherwise. The existence of this limit follows from the
subadditivity in $m$ of $\cl(w^{m})$. Stable commutator length is
an important object in geometric group theory and the theory of the
free group: see the book of Calegari \cite{calegari2009scl}. One
of the fundamental results about $\scl$ is due to Calegari \cite{CALRATIONAL}:
\begin{thm}[Calegari]
If $w\in\F$ then $\scl(w)\in\Q\cup\{\infty\}.$
\end{thm}

The function $\scl:\F\to\Q\cup\{\infty\}$ takes on infinitely many
values when $r\geq2$. For example, it is a result of Culler \cite{CULLER}
that $\cl([x_{1},x_{2}]^{n})=\lfloor\frac{n}{2}\rfloor+1$ and hence
\begin{align*}
\scl([x_{1},x_{2}]^{k}) & =\lim_{m\to\infty}\frac{\cl([x_{1},x_{2}]^{km})}{m}=\lim_{m\to\infty}\frac{\lfloor\frac{km}{2}\rfloor+1}{m}=\frac{k}{2}.
\end{align*}
It is even known, by Calegari \cite{calegariSails}, that if $r\geq4$,
$\scl(\F_{r})$ contains a rational with any given denominator. The
following theorem is proved in \cite[Cor. 1.11]{MP2}.
\begin{thm}[Magee-Puder]
\label{thm:scl-sep}For $w\in\F$, knowing the word measure $\mu_{w}$
on every $\U(n)$ determines $\scl(w)$. As a consequence, for $w_{1},w_{2}\in\F$,
\[
w_{1}\stackrel{\mathbf{CptGrp}}{\sim}w_{2}\implies\scl(w_{1})=\scl(w_{2}).
\]
\end{thm}

The proof of Theorem \ref{thm:scl-sep} can be reinterpreted as the
establishment of the following equality:
\[
\scl(w)=-\frac{1}{2}\sup_{k\geq0}\lim_{n\to\infty}\frac{\log\left|\tau_{\U(n),\Sym^{k}(\std)}(w)\right|}{\log\left(n^{k}\right)}
\]
 where $\Sym^{k}(\std)$ is the symmetric $k^{th}$ power of the standard
representation of $\U(n)$. 

\section{The GNS construction}

In this section and the next one, we allow $\Gamma$ to be any countable
discrete group. The collection of positive definite functions $\tau$
on $\Gamma$ form a convex cone that we will denote by $\P(\Gamma)$.
The importance of positive definite functions on groups comes from
their role in the Gelfand-Naimark-Segal (GNS) construction \cite{GelfandNaimark43,Segal47}:
\begin{thm}[GNS construction]
If $\tau:\Gamma\to\C$ is a positive definite function with $\tau(e)=1$,
then there is a \textbf{GNS triple} $(\pi_{\tau},\H_{\tau},\xi_{\tau})$
where
\begin{itemize}
\item $\H_{\tau}$ is a Hilbert space with inner product $\langle\bullet,\bullet\rangle_{\tau}$,
\item $\pi_{\tau}:\Gamma\to U(\H_{\tau})$ is a homomorphism from $\Gamma$
to the group of unitary operators $U(\H_{\tau})$ on $\H_{\tau}$,
\item $\xi_{\tau}\in\H_{\tau}$ is a unit cyclic vector for the unitary
representation $\pi_{\tau}$, meaning that the linear span of $\{\,\pi_{\tau}(g)\xi_{\tau}\,:\,g\in\Gamma\,\}$
is dense in $\H_{\tau}$, and
\item we have $\tau(g)=\langle\pi_{\tau}(g)\xi_{\tau},\xi_{\tau}\rangle_{\tau}$
for all $g\in\Gamma$. 
\end{itemize}
The GNS triple associated to $\tau$ is unique up to unitary equivalence:
if $(\pi_{\tau},\H_{\tau},\xi_{\tau})$ and $(\pi'_{\tau},\H'_{\tau},\xi'_{\tau})$
are two GNS triples, then there is a unitary intertwiner $u:\H_{\tau}\to\H'_{\tau}$
such that $u(\xi_{\tau})=\xi'_{\tau}$ and for all $g\in\Gamma$,
$\pi_{\tau}(g)=u^{-1}\pi'_{\tau}(g)u$.

Conversely, if $\pi:\Gamma\to U(\H)$ is a unitary representation
of $\Gamma$ on a Hilbert space $\H$, and $\xi\in\H$ is a unit vector,
then 
\[
\tau(w)=\langle\pi(w)\xi,\xi\rangle
\]
 is a positive definite function on $\Gamma$ with $\tau(e)=1$. Moreover,
if $\xi$ is cyclic then the GNS triple $(\pi_{\tau},\H_{\tau},\xi_{\tau})$
is equivalent, in the sense described above, to $(\pi,\H,\tau)$.
\end{thm}

\begin{example}[Regular representation]
The function $\tau_{\lambda}$ introduced in Example \ref{exa:reg-defn}
is obtained as a matrix coefficient in the \emph{regular representation
}of $\F$. Indeed let $\lambda:\F\to U(\ell^{2}(\F))$ denote the
left regular representation. Then 
\[
\tau_{\lambda}(w)=\langle\lambda(w)\delta_{e},\delta_{e}\rangle=\delta_{we}\quad\forall w\in\F.
\]
Conversely, the GNS triple associated to $\lambda$ is, up to isomorphism,
$(\lambda,\ell^{2}(\F),\delta_{e})$. 
\end{example}

\begin{example}[Trivial representation]
The function $\tau_{\triv}$ introduced in Example \ref{exa:triv-definition}
is obtained as a matrix coefficient in the trival representation $\triv:\F\to U(\C)$.
Indeed, $1\in\C$ is a cyclic vector for this representation. Thus
\[
\tau_{\triv}(w)=\langle\triv(w)1,1\rangle=\langle1,1\rangle=1\quad\forall w\in\F.
\]
\end{example}

\begin{example}[Compact group construction, continued]
\label{exa:compact-group-construction II}Let $\tau_{G,\pi}$ be
the positive definite function constructed in Example \ref{exa:compact-group-construction}.
Recall that $\pi:G\to U(V)$ is a finite dimensional unitary representation
of the compact group $G$. In this case, $\tau_{G,\pi}$ arises as
a matrix coefficient in a subrepresentation of the direct integral\footnote{The direct integral of representations is a generalization of the
direct sum that uses a topological space with a Borel measure to index
the summation, instead of a discrete set. For details see \cite[\S 2.4]{Mackey}.} 
\[
\Pi_{G,\pi}=\int_{G^{r}}^{\oplus}\Pi_{\mathbf{g}}d\mu^{r}(\mathbf{g})
\]
where $\Pi_{\mathbf{g}}:\F\to U(\End(V))$, $\Pi_{\mathbf{g}}(w).A=\pi(w(\mathbf{g}))A$.
The inner product on $\End(V)$ is given by $\langle A,B\rangle=\trace(AB^{*})$.
The representation is generated by the cyclic vector
\[
\xi=\int_{G^{r}}^{\oplus}\mathrm{Id}_{\End(V)}d\mu^{r}(\mathbf{g})
\]
and $\H$ is the closed linear span of $\{\Pi_{G,\pi}(w)\xi\,:\,w\in\F\}$.
Then we have
\[
\langle\Pi_{G,\pi}(w)\xi,\xi\rangle=\int_{G^{r}}\trace(\pi(w(\mathbf{g})))d\mu^{r}(\mathbf{g})=\tau_{G,\pi}(w).
\]
\end{example}

\begin{example}[Characteristic subgroup construction]
\label{exa:char-subgroup}We now turn to yet another type of examples
of $\Aut(\F)$-invariant positive definite functions on $\F$. Let
$\Lambda\leq\F$ be a \emph{characteristic subgroup}, meaning that
$\alpha(\Lambda)=\Lambda$ for any $\alpha\in\Aut(\F)$. As conjugation
by elements of $\F$ gives automorphisms, $\Lambda$ is necessarily
normal in $\F$. Some examples of characteristic subgroups of $\F$
include
\begin{itemize}
\item The commutator subgroup $[\F,\F]$.
\item Groups in the derived series of $\F$, for example, $\left[[\F,\F],[\F,\F]\right]$.
\item Groups in the lower or upper central series of $\F$.
\item If $H$ is any group, the intersection of all kernels of homomorphisms
$\F\to H$ (this may be trivial).
\end{itemize}
Note that $\Aut(\F)$ acts by automorphisms on the group $\F/\Lambda$.
We obtain an $\Aut(\F)$-invariant positive definite function on $\F$
denoted by $\tau_{\Lambda}$ and given by
\[
\tau_{\Lambda}(w)=\delta_{w\Lambda,e}=\begin{cases}
1 & \text{if}~\ensuremath{w}\ensuremath{\in\Lambda}\\
0 & \text{if}~\ensuremath{w\notin\Lambda}.
\end{cases}
\]
Indeed, $\tau_{\Lambda}$ arises from the GNS triple $(\pi,\H,\xi)$
where $\H=\ell^{2}(\F/\Lambda)$, $\pi$ is the quasi-regular representation,
and $\xi=\delta_{e}\in\ell^{2}(\F/\Lambda)$. It is clear that $\tau_{\Lambda}$
is $\Aut(\F)$-invariant since $\Lambda$ is characteristic in $\F$.\footnote{More generally, if $T$ is any positive definite function on $\F/\Lambda$,
then $\tau(w)\stackrel{\df}{=}T(w\Lambda)$ will be a positive definite
function on $\F$. It will be $\Aut(\F)$-invariant if $T$ is invariant
under the induced action of $\Aut(\F)$ on $\F/\Lambda$. Since classifying
these $T$ in general seems hard, we do not pursue this in detail
here.}
\end{example}

Evidently, the subgroup of $\F$ generated by an $\Aut(\F)$-orbit
is characteristic. Therefore, the construction of Example \ref{exa:char-subgroup}
allows us to make a little progress on Question \ref{que:Do--invariant-positive-functions-sep-orbits}.
\begin{prop}
\label{prop:char-subgroup-orbits}$\star$ If $w_{1},w_{2}\in\F$
and the orbits $\Aut(\F).w_{1}$ and $\Aut(\F).w_{2}$ generate different
subgroups of $\F$ then
\[
w_{1}\stackrel{\mathbf{PosDef}}{\not\sim}w_{2}.
\]
\end{prop}

We stress, however, that in general different $\Aut(\F)$-orbits in
$\F$ may generate the same subgroup. This is illustrated in the following
two examples:
\begin{example}
While $\Aut(\F).w$ and $\Aut(\F).w^{-1}$ generate the same subgroup,
there is no reason for $w$ and $w^{-1}$ to be in the same orbit.
For example, $w=x^{2}y^{2}xy^{-1}$ is not in the same orbit as its
inverse.
\end{example}

\begin{example}
Let $r=2$ and $w=x_{1}^{2}x_{2}^{3}$. Let $\Lambda$ be the group
generated by $\Aut(\F_{2}).w$. Then $\Lambda$ also contains $w'=x_{2}^{-2}x_{1}^{3}$,
since $(x_{1},x_{2})\mapsto(x_{2}^{-1},x_{1})$ is in $\Aut(\F_{2})$.
However,
\[
ww'=x_{1}^{2}x_{2}x_{1}^{3}
\]
is in $\Lambda$, and is primitive, since $x_{1}^{2}x_{2}x_{1}^{3}$
and $x_{1}$ generate $\F_{2}$. Since $\Lambda$ is characteristic,
all primitive elements must be in $\Lambda$, and in particular, $x_{1}$
and $x_{2}$ are in $\Lambda$, so $\Lambda=\F_{2}$. However, $w$
itself is not primitive: this can be inferred from Whitehead algorithm
\cite[Chapter I.4]{Lyndon1977}, or from the fact that $\tau_{S_{3},\mathrm{std}}\left(w\right)=1.5\ne1=\tau_{S_{n},\mathrm{std}}\left(x_{1}\right)$.
I.e., $\Aut(\F).w\neq\Aut(\F).x_{1}$, but $\Aut(\F).x_{1}$ and $\Aut(\F).w$
generate the same group.
\end{example}

Before moving on, we address the following question. What does the
GNS construction tell us about $\Aut(\Gamma)$-invariant positive
definite functions? We will denote by $\P_{1}(\Gamma)$ the elements
$\tau$ of $\P(\Gamma)$ with $\tau(e)=1$. Suppose that $\tau$ is
an element of $\P_{1}(\Gamma)^{\Aut(\Gamma)}$, the elements of $\P_{1}(\Gamma)$
which are invariant under $\Aut(\Gamma)$. In this case, one can extend
$\tau$ to a positive definite function $\tau^{+}$ on the semidirect
product $\Gamma\rtimes\Aut(\Gamma)$ by the formula
\[
\tau^{+}(\gamma,\alpha)\stackrel{\df}{=}\tau(\gamma).
\]

\begin{lem}
\label{lem:pos-def-on-semidirect-product}$\star$ If $\tau\in\P_{1}(\Gamma)^{\Aut(\Gamma)}$,
then $\tau^{+}$ is a positive definite function on $\Gamma\rtimes\Aut(\Gamma)$.
\end{lem}

Let $(\pi_{\tau^{+}},\H_{\tau^{+}},\xi_{\tau^{+}})$ be the associated
GNS triple to $\tau^{+}$. Since $\langle\pi(e,\alpha)\xi_{\tau^{+}},\xi_{\tau^{+}}\rangle=1$,
the cyclic vector $\xi_{\tau^{+}}$ is an invariant vector for the
embedded copy of $\Aut(\Gamma)$ in $\Gamma\rtimes\Aut(\Gamma)$ under
$\alpha\mapsto(e,\alpha)$. The map $\tau\mapsto\tau^{+}$ gives a
linear embedding of $\P_{1}(\Gamma)^{\Aut(\Gamma)}$ into $\P_{1}(\Gamma\rtimes\Aut(\Gamma)).$ 

\begin{example}[Compact group construction, continued]
\label{exa:compact-group-construction-III}Recall the notations of
Examples \ref{exa:compact-group-construction} and \ref{exa:compact-group-construction II}.
Let $\H_{0}$ denote the Hilbert space
\[
\H_{0}=\int_{G^{r}}^{\oplus}\End(V)d\mu^{r}(\mathbf{g}).
\]
We will describe a unitary representation of $\F\rtimes\Aut(\F)$
on this Hilbert space as follows. A vector in $\H_{0}$ is (an equivalence
class) of an $L^{2}$ function $\mathbf{g}\mapsto B_{\mathbf{g}}$
for $\mathbf{g}\in G^{r}$ and \textbf{$B_{\mathbf{g}}\in\End(V)$.
}We define for $(w,\alpha)\in\F\rtimes\Aut(\F)$
\[
\Pi_{0}(w,\alpha)\{\mathbf{g}\mapsto B_{\mathbf{g}}\}=\{\mathbf{g\mapsto}\pi(w(\mathbf{g}))B_{\alpha^{-1}(\mathbf{g})}\}.
\]
It is straightforward to check this this gives a unitary representation
of $\F\rtimes\Aut(\F)$ on $\H_{0}$, using Lemma \ref{lem:Haar-measure-Aut-invariant}.
Now let $\Pi_{G,\pi}^{+},\H_{G,\pi}^{+}$ be the subrepresentation
of $\Pi_{0}$ generated by the vector 
\[
\xi_{G,\pi}^{+}=\frac{1}{\sqrt{\dim V}}\int_{G^{r}}^{\oplus}\mathrm{Id}_{\End(V)}d\mu^{r}(\mathbf{g}).
\]
Let $\tilde{\tau}_{G,\pi}=\frac{1}{\dim V}\tau_{G,\pi}\in\P_{1}(\F)^{\Aut(\F)}$.
Now one has 
\[
\tilde{\tau}_{G,\pi}^{+}(w,\alpha)=\langle\Pi_{0}(w,\alpha)\xi_{G,\pi}^{+},\xi_{G,\pi}^{+}\rangle=\frac{1}{\dim V}\int_{G^{r}}\trace(\pi(w(\mathbf{g})))d\mu^{r}(\mathbf{g})=\tilde{\tau}_{G,\pi}(\alpha).
\]
Thus we have constructed an explicit model for the GNS triple associated
to $\tilde{\tau}_{G,\pi}^{+}$.
\end{example}

\section{Extremal functions}

To study $\P(\Gamma)$ it is convenient to introduce an operator algebra.
We begin with $\C[\Gamma]$, the group algebra of $\Gamma$. We define
a norm on $\C[\Gamma]$ by 
\[
\|a\|=\sup_{\pi}\|\pi(a)\|
\]
where $\pi$ ranges over all cyclic $*$-representations\footnote{A \emph{$*$-representation} $(\pi,V)$ of $\C[\Gamma]$ consists
of a Hilbert space $V$ and a $\C$-algebra homomorphism $\pi$ from
$\C[\Gamma]$ to the bounded endomorphisms $B(V)$ of $V$ that also
respects the star operations. The star operation on $\C[\Gamma]$
takes $\sum a_{\gamma}\gamma$ to $\sum\overline{a_{\gamma}}\gamma^{-1}$
and the star operation on $B(V)$ is conjugate transpose. The $*$-representation
$(\pi,V)$ is \emph{cyclic} if $V$ contains a vector $v$ such that
$\pi(\C(\Gamma)).v$ is dense in $V$.} of $\C[\Gamma]$. The completion of $\C[\Gamma]$ with respect to
this norm is a $C^{*}$-algebra called the (full) group $C^{*}$-algebra
of $\Gamma$, denoted by $C^{*}(\Gamma)$. 

Any $\tau\in\P(\Gamma)$ extends to a continuous linear functional
$\tau$ on $C^{*}(\Gamma)$ with $\|\tau\|=\tau(e)$. Therefore $\P_{1}(\Gamma)$
linearly embeds into the unit ball of the linear dual of $C^{*}(\Gamma)$.
The set $\P_{1}(\Gamma)$ is closed in the weak-$*$ topology and
hence by the Banach-Alaoglu Theorem, $\P_{1}(\Gamma)$ is weak-$*$
compact. Since $\P_{1}(\Gamma)$ is also obviously convex, the Krein-Milman
Theorem tells us that $\P_{1}(\Gamma)$ is the (weak-$*$) closed
convex hull of its extreme points that we will denote by $\overline{\hull}\ext[\P_{1}(\Gamma)]$.
The classical relevance of the extreme points is the following result
from \cite{Segal47}:
\begin{thm}
\label{thm:extreme-implies-irreducible}For $\tau\in\P_{1}(\Gamma)$,
$\tau\in\ext[\P_{1}(\Gamma)]$ if and only if the GNS representation
$\pi_{\tau}$ is irreducible.
\end{thm}

We may improve on the fact that $\P_{1}(\Gamma)=\overline{\hull}\ext[\P_{1}(\Gamma)]$
by means of Choquet theory. Since $\Gamma$ is countable, $C^{*}(\Gamma)$
is separable, so $\P_{1}(\Gamma)$ is metrizable. Choquet's Theorem
\cite[pg.\;14]{Phelps} gives in the current context the following.
\begin{thm}[Choquet's Theorem for $\P_{1}(\Gamma)$]
\label{thm:choquet}If $\tau\in\P_{1}(\Gamma)$, there is a (regular)
Borel probability measure $\nu_{\tau}$ supported on $\ext[\P_{1}(\Gamma)]$
such that for any $g\in\Gamma$,
\[
\tau(g)=\int\tilde{\tau}(g)d\nu_{\tau}(\tilde{\tau}).
\]
In this case, we say that $\nu_{\tau}$ \textbf{represents $\tau$.}
\end{thm}

Recall we have seen as a consequence of the Krein-Milman Theorem that
$\P_{1}(\Gamma)=\overline{\hull}\ext[\P_{1}(\Gamma)]$. Note that
$\P_{1}(\Gamma)^{\Aut(\Gamma)}$ is a weak-$*$ closed subset of $\P_{1}(\Gamma)$,
since it is the intersection over $\alpha\in\Aut(\Gamma)$ and $g\in\Gamma$
of the sets of $\tau\in\P_{1}(\Gamma)$ such that 
\[
\tau(\alpha(g))-\tau(g)=0,
\]
each of which is the vanishing locus of a weak-$*$ continuous function
on $\P_{1}(\Gamma)$. Hence $\P_{1}(\Gamma)^{\Aut(\Gamma)}$ is compact,
and also convex, so the Krein-Milman Theorem gives
\[
\P_{1}(\Gamma)^{\Aut(\Gamma)}=\overline{\hull}\ext[\P_{1}(\Gamma)^{\Aut(\Gamma)}].
\]
\emph{This reduces Question \ref{que:Do--invariant-positive-functions-sep-orbits}
to the question of whether the functions in $\ext[\P_{1}(\F)^{\Aut(\F)}]$
separate $\Aut(\F)$-orbits. }It also raises the interesting question
of when our known examples of elements of $\P_{1}(\F)^{\Aut(\F)}$
are extremal.
\begin{thm}
\label{thm:extreme-functions-and-ergodic-actions}$\star$ Recall
the notations from Example \ref{exa:compact-group-construction}.
Let $\pi$ be an irreducible unitary representation of the compact
group $G$. Then the function $\tilde{\tau}_{G,\pi}=\frac{1}{\dim V}\tau_{G,\pi}$
is in $\ext[\P_{1}(\F)^{\Aut(\F)}]$ if and only if the action by
precomposition of $\Aut(\F)$ on $G^{r}\cong\Hom(\F,G)$ is ergodic
with respect to the Haar measure $\mu^{r}$. 
\end{thm}

Fortunately, the action of $\Aut(\F)$ on $G^{r}$ has already been
investigated by different researchers. The following theorem was proved
by Goldman \cite{GOLDMAN} when $G$ is a Lie group with simple factors
of type $\mathrm{U}(1)$ or $\mathrm{SU}(2)$, and extended by Gelander
in \cite{Gelander} to the following.
\begin{thm}[Goldman, Gelander]
\label{thm:gelander}Let $G$ be a compact connected semisimple Lie
group and suppose that $r\geq3$. Then the action of $\Aut(\F_{r})$
on $G^{r}$ is ergodic with respect to the Haar measure $\mu^{r}$.
\end{thm}

Theorem \ref{thm:gelander} together with Theorem \ref{thm:extreme-functions-and-ergodic-actions}
allow us to produce many elements of $\ext[\P_{1}(\F)^{\Aut(\F)}]$
using compact groups. The situation for finite groups is less clear.
One important point is that when $G$ is a finite non-trivial group,
the action of $\Aut(\F)$ on $G^{r}$ will never be ergodic with respect
to the Haar measure $\mu^{r}$. The reason is that the subset
\[
\Epi(\F,G)=\{\,\phi\in\Hom(\F,G)\::\:\phi(\F)=G\,\}\subset\Hom(\F,G)
\]
is clearly invariant, and its complement has positive measure. Nonetheless,
one could alter the definitions of $\tilde{\tau}_{G,\pi}$ to use
the uniform measure on $\Epi(\F,G)$ in place of $\mu^{r}$. If $\Aut(\F)$
acts transitively on $\Epi(\F,G)$, this will yield elements of $\ext[\P_{1}(\F)^{\Aut(\F)}]$.
However, it is a well-known open problem whether this is the case
even for simple $G$:
\begin{conjecture}[Wiegold's conjecture]
\label{conj:Wiegold}If $G$ is a finite simple group, and $r\geq3$,
then $\Aut(\F_{r})$ acts transitively on $\Epi(\F_{r},G)$.
\end{conjecture}

The reader is invited to see the article of Lubotzky \cite{Lubotzky}
for a survey of Wiegold's conjecture and related questions. We also
mention that it is proved in \cite{Hanany} that two words induce
the same measure on every finite group if and only if they induce
the same measure on every finite group via epimorphisms.

\section{A toy problem}

One of the philosophies of Voiculescu's Free Probability Theory introduced
in \cite{Voiculescu1991} is that one passes from classical probability
problems involving commuting random variables to problems involving
non-commutative random variables \cite{VDNbook,NSbook,MSbook}. In
the same spirit, we may view the setup of the current paper as arising
from a process by which one replaces
\begin{align*}
\Z^{r} & \rightsquigarrow\F\\
\Aut(\Z^{r})=\GL_{r}(\Z) & \rightsquigarrow\Aut(\F).
\end{align*}
In the setting of $\GL_{r}(\Z)$ acting on $\Z^{r}$, we understand
all the questions of this paper, and as we will see, they are connected
to classical results concerning Borel measures on tori that are instructive
to recall.

First we consider the extreme points of $\P_{1}(\Z^{r})$. If $\tau\in\ext[\P_{1}(\Z^{r})]$,
then the associated GNS triple $(\pi_{\tau},\H_{\tau},\xi_{\tau})$
has $\pi_{\tau}$ irreducible, so as $\Z^{r}$ is abelian, $\H_{\tau}$
is one-dimensional, and $\langle\pi_{\tau}(\underline{x})\xi_{\tau},\xi_{\tau}\rangle=\exp(2\pi i\theta^{\tau}.\underline{x})$
for some 
\begin{align*}
\theta^{\tau} & =(\theta_{1}^{\tau},\ldots,\theta_{r}^{\tau})\in[0,1)^{r},
\end{align*}
where $t_{\tau}.\underline{x}$ is the standard scalar (dot) product.
Hence the correspondence $\tau\mapsto\theta^{\tau}$ identifies $\ext[\P_{1}(\Z^{r})]$
with the torus $\T^{r}=(S^{1})^{r}$. The weak-$*$ topology on $\ext[\P_{1}(\Z^{r})]$
corresponds to the standard metric topology on $\T^{r}$.

By Choquet's Theorem (Theorem \ref{thm:choquet}) in this context,
there is a regular Borel measure $\nu_{\tau}$ on $\ext[\P_{1}(\Z^{r})]=\T^{r}$
such that for any $\underline{x}\in\Z^{r}$

\[
\tau(\underline{x})=\int_{\ext[\P_{1}(\Z^{r})]}\tilde{\tau}(\underline{x})d\nu_{\tau}(\tilde{\tau})=\int_{\T^{r}}\exp(2\pi i\theta.\underline{x})d\nu_{\tau}(\theta).
\]
In other words, $\tau(\underline{x})$ is simply the Fourier transform
of $\nu_{\tau}$ evaluated at $\underline{x}$. 

In this case, as $\Z^{r}$ is abelian, it is a consequence of the
Stone-Weierstrass Theorem that $\nu_{\tau}$ is uniquely determined
by $\tau$. It now follows that if $\tau$ is $\GL_{r}(\Z)$-invariant,
so too is $\nu_{\tau}$. This reduces the classification of $\GL_{r}(\Z)$-invariant
positive definite functions on $\Z^{r}$ to the classification of
$\GL_{r}(\Z)$-invariant Borel probability measures on $\T^{r}$.
Moreover, the extreme points $\ext[\P_{1}(\Z^{r})^{\GL_{r}(\Z)}]$
correspond to extremal invariant measures, which by standard facts
\cite[Prop 12.4]{Phelps} are the ergodic ones. One has the following
classification of such measures by Burger \cite[Prop. 9]{BURGER}\footnote{Although \cite[Prop. 9]{BURGER} states the result for $\mathrm{SL}_{r}(\Z)$,
it also holds for $\GL_{r}(\Z)$.}.
\begin{prop}
\label{prop:classification-of-ergodic-Borel-meuase}Any $\GL_{r}(\Z)$-invariant
ergodic Borel probability measure on $\T^{d}$ is either Lebesgue
measure, or atomic and supported on a finite $\GL_{r}(\Z)$-orbit.
\end{prop}

This can be read as a full classification of $\ext[\P_{1}(\Z^{r})^{\GL_{r}(\Z)}]$.
While an analogous classification of $\ext[\P_{1}(\F)^{\Aut(\F)}]$
seems out of reach, it suggests that it would be interesting to pursue
(see $\S$\ref{sec:Further-Open-Questions}). Even further, we can
show the following.
\begin{thm}
\label{thm:hierarchy-collapse}$\star$ For $\Z^{r}$, $\GL_{r}(\Z)$,
in place of $\F$, $\Aut(\F)$, the hierarchy in (\ref{eq:relations-impliications})
completely collapses. More concretely, for $\underline{x}=(x_{1},\ldots,x_{r})$,
$\underline{y}=(y_{1},\ldots,y_{r})\in\Z^{r}$, $\underline{x}\in\GL_{r}(\Z).\underline{y}$
if and only if there is a finite abelian group $G$ with uniform measure
$\mu$ such that $\mu_{\underline{x}}=\mu_{\underline{y}}$, where
e.g. $\mu_{\underline{x}}=\underline{x}_{*}\mu^{r}$ is the pushforward
of $\mu^{r}$ on $G^{r}$ under the map 
\[
\underline{x}:(g_{1},\ldots,g_{r})\mapsto x_{1}g_{1}+\cdots+x_{r}g_{r}.
\]
\end{thm}

\section{Further open questions\label{sec:Further-Open-Questions}}

Our discussion above leads to a possible alternative approach to Conjectures
\ref{que:compact-groups-separate-orbits} and \ref{que:finite-groups-separate-orbits}.
This consists of the following program:
\begin{description}
\item [{I}] Resolve Question \ref{que:Do--invariant-positive-functions-sep-orbits},
i.e. show that the elements of $\ext[\P_{1}(\F)^{\Aut(\F)}]$ separate
$\Aut(\F)$-orbits.
\item [{II}] Prove that the elements of $\ext[\P_{1}(\F)^{\Aut(\F)}]$
can be approximated in a suitable way by elements arising from finite
or compact groups via the construction given in Example \ref{exa:compact-group-construction}.
\end{description}
Whether or not step II above can be accomplished is of independent
interest. The following question is enticing:

\begin{question}
\label{que:classification-of-extremal-measures}Is it possible to
classify the elements of $\ext[\P_{1}(\F)^{\Aut(\F)}]$ in a way that
generalizes Proposition \ref{prop:classification-of-ergodic-Borel-meuase}?
\end{question}

As mentioned above, Question \ref{que:classification-of-extremal-measures}
may be very hard or impossible. It would be nice to reduce Question
\ref{que:classification-of-extremal-measures} to a question about
the classification of $\Aut(\F)$-invariant ergodic measures as in
Proposition \ref{prop:classification-of-ergodic-Borel-meuase}. The
problem with this is that the measure on $\ext[\P_{1}(\F)]$ that
represents an element of $\ext[\P_{1}(\F)^{\Aut(\F)}]$, given by
Theorem \ref{thm:choquet}, may not be unique; however we do not know
whether this is the case in practice. Therefore one has the technical
question:

\begin{question}
\label{que:uniqueness of Choquet measure}Is there some $\tau\in\ext[\P_{1}(\F)^{\Aut(\F)}]$
that is not represented by a unique regular Borel probability measure
$\nu_{\tau}$ supported on $\ext[\P_{1}(\F)]$?
\end{question}

Setting aside the technical issue presented in Question \ref{que:uniqueness of Choquet measure},
one can still ask about the classification of $\Aut(\F)$-invariant
ergodic measures.
\begin{question}
\label{que:Classify-the-Borel-measures-on-extreme-points}Classify
the Borel probability measures supported on $\ext[\P_{1}(\F)]$ that
are invariant and ergodic for the action of $\Aut(\F)$.
\end{question}

One specific instance of Question \ref{que:Classify-the-Borel-measures-on-extreme-points}
that is much more approachable is the following.
\begin{question}
\label{que:classification-of-invariant-borel-measures}Let $G$ be
a compact topological group. For simplicity, one might like to assume
that $G$ is a connected compact semisimple Lie group. What are the
$\Aut(\F)$-invariant and ergodic Borel measures on $G^{r}$?
\end{question}

Note that Theorem \ref{thm:gelander} classifies, under certain hypotheses,
the $\Aut(\F)$-invariant and ergodic Borel measures on $G^{r}$ that
are absolutely continuous with respect to the Haar measure, and Question
\ref{que:classification-of-invariant-borel-measures} removes this
assumption.

Short of classification results, one may hope for other statements
that would accomplish step II above. For example,

\begin{question}
\label{que:approximation-by-compact-groups}Is it possible that the
weak-$*$ closure of the functions $\tilde{\tau}_{G,\pi}$ (cf. Examples
\ref{exa:compact-group-construction}, \ref{exa:compact-group-construction II},
\ref{exa:compact-group-construction-III}) contains $\ext[\P_{1}(\F)^{\Aut(\F)}]$?
\end{question}

Again, Question \ref{que:approximation-by-compact-groups} may be
very difficult. However, considering Question \ref{que:approximation-by-compact-groups}
leads us to realize that we do not even know very basic things about
$\ext[\P_{1}(\F)^{\Aut(\F)}]$. Note that by (\ref{eq:schur-orth}),
all the examples of elements $\tau\in\ext[\P_{1}(\F)^{\Aut(\F)}]$
given in this paper, other than $\tau_{\triv}$, have the property
that $\tau(x_{1})=0$. This invites the following basic and intriguing
question.
\begin{question}
Is there a $\tau\in\ext[\P_{1}(\F)^{\Aut(\F)}]$ with $\tau\neq\tau_{\triv}$
such that $\tau(x_{1})\neq0$?
\end{question}

Also with Question \ref{que:approximation-by-compact-groups} in mind,
if $\tau$ is a weak-$*$ limit of functions $\tilde{\tau}_{G_{i},\pi_{i}}$
with $\dim(\pi_{i})\to\infty$ as $i\to\infty$, then by Theorem \ref{thm:frob},
$\tau([x_{1},x_{2}])=0$. This suggests that it might be helpful to
ask the converse.

\begin{question}
If $r\geq2$ and $\tau\in\ext[\P_{1}(\F)^{\Aut(\F)}]$ with $\tau([x_{1},x_{2}])=0,$
is $\tau$ a weak-$*$ limit of the functions $\tilde{\tau}_{G,\pi}$?
\end{question}

Finally, turning to Question \ref{que:Do--invariant-positive-functions-sep-orbits}
in view of step I above, we propose the following.
\begin{question}
Find new constructions of $\Aut(\F)$-invariant positive definite
functions on $\F$.
\end{question}

\appendix

\section{Proofs of background results}

In some of our proofs we use the following simple fact.
\begin{lem}
\label{lem:simple-lemma}If $G$ is a compact topological group with
probability Haar measure $\mu$, $(\pi,V)$ is an irreducible unitary
representation of $G$, and $A\in\End(V)$, then
\[
\int_{G}\pi(g)A\pi(g)^{-1}d\mu(g)=\frac{\trace(A)}{\dim V}\mathrm{Id}_{V}.
\]
\end{lem}

\begin{proof}
The left hand side is invariant under conjugation by elements $\pi(g)$
with $g\in G$, so by Schur's Lemma is a scalar multiple of the identity.
The trace of the matrices inside the integral is constant and equal
to $\mathrm{tr}(A)$, and so the result of the integral is a scalar
multiple of the identity with trace $\mathrm{tr}(A).$
\end{proof}
\begin{proof}[\textbf{Proof of Lemma \ref{lem:Haar-measure-Aut-invariant}}]
It is enough to show that $\mu^{r}$ is invariant under the Nielsen
generators given in (\ref{eq:nielsen-perm}), (\ref{eq:nielsen-inverse}),
(\ref{eq:nielsen-mult}). The measure $\mu^{r}$ is determined by
the formula, for any continuous $f:G^{r}\to\C$,
\[
\int_{G^{r}}f(\mathbf{g})d\mu^{r}(\mathbf{g})=\int_{G}\ldots\int_{G}f(g_{1},\ldots,g_{r})d\mu(g_{1})\ldots d\mu(g_{r}).
\]
For $\sigma\in S_{r}$ we have 
\begin{align*}
\int_{G^{r}}f(\alpha_{\sigma}(\mathbf{g}))d\mu^{r}(\mathbf{g}) & =\int_{G}\ldots\int_{G}f(g_{\sigma(1)},\ldots,g_{\sigma(r)})d\mu(g_{1})\ldots d\mu(g_{r})\\
 & =\int_{G}\ldots\int_{G}f(g_{1},\ldots,g_{r})d\mu(g_{1})\ldots d\mu(g_{r})\\
 & =\int_{G^{r}}f(\mathbf{g})d\mu^{r}(\mathbf{g})
\end{align*}
by Fubini's Theorem. We have
\begin{align*}
\int_{G^{r}}f(\iota(\mathbf{g}))d\mu^{r}(\mathbf{g}) & =\int_{G}\ldots\int_{G}f(g_{1}^{-1},\ldots,g_{r})d\mu(g_{1})\ldots d\mu(g_{r})\\
 & =\int_{G}\ldots\int_{G}f(g_{1},\ldots,g_{r})d\mu(g_{1})\ldots d\mu(g_{r})\\
 & =\int_{G^{r}}f(\mathbf{g})d\mu^{r}(\mathbf{g})
\end{align*}
since $\mu$ is invariant under pushforward by $g\mapsto g^{-1}$
(that is a result of the bi-invariance and uniqueness of Haar measure).
Finally, we have
\begin{align*}
\int_{G^{r}}f(\gamma(\mathbf{g}))d\mu^{r} & =\int_{G}\ldots\int_{G}f(g_{1}g_{2},\ldots,g_{r})d\mu(g_{1})\ldots d\mu(g_{r})\\
 & =\int_{G}\ldots\int_{G}f(g_{1},\ldots,g_{r})d\mu(g_{1})\ldots d\mu(g_{r})\\
 & =\int_{G^{r}}f(\mathbf{g})d\mu^{r}(\mathbf{g})
\end{align*}
by the right-invariance of Haar measure.
\end{proof}
\begin{proof}[\textbf{Proof of Lemma \ref{lem:W-MEASURE-determined-by-character-integrals}}]
Let $C(G)$ denote the Banach space of continuous complex valued
functions on $G$ with supremum norm. Since $G^{r}$ and $G$ are
compact and Hausdorff, and $w:G^{r}\to G$ is continuous, $\mu^{r}$
is a regular Borel probability measure, and so too is the pushforward
measure $\mu_{w}=w_{*}\mu^{r}$. Hence by the Riesz-Markov Theorem
$\mu_{w}$ is uniquely determined by the formula
\[
\int_{g\in G}f(g)d\mu_{w}(g)=\int_{(g_{1},\ldots,g_{r})\in G^{r}}f(w(g_{1},\ldots,g_{r}))d\mu^{r}(g_{1},\ldots,g_{r}),\quad\forall f\in C(G).
\]
Since the linear span of matrix coefficients of irreducible unitary
representations is dense in $C(G)$ by the Peter-Weyl Theorem, it
follows that $\mu_{w}$ is determined by the integrals 
\[
\int_{(g_{1},\ldots,g_{r})\in G^{r}}\langle\pi(w(g_{1},\ldots,g_{r}))v_{1},v_{2}\rangle d\mu^{r}(g_{1},\ldots,g_{r}).
\]
where $\pi:G\to U(V)$ is an irreducible unitary representation of
$G$ and $v_{1},v_{2}\in V$. On the other hand, we have
\begin{align*}
 & \int_{(g_{1},\ldots,g_{r})\in G^{r}}\left\langle \pi(w(g_{1},\ldots,g_{r}))v_{1},v_{2}\right\rangle d\mu^{r}(g_{1},\ldots,g_{r})\\
 & =\int_{h\in G}\int_{(g_{1},\ldots,g_{r})\in G^{r}}\langle\pi(w(hg_{1}h^{-1},\ldots,hg_{r}h^{-1}))v_{1},v_{2}\rangle d\mu^{r}(g_{1},\ldots,g_{r})d\mu(h)\\
 & =\int_{h\in G}\int_{(g_{1},\ldots,g_{r})\in G^{r}}\langle\pi(h)\pi(w(g_{1},\ldots,g_{r}))\pi(h)^{-1}v_{1},v_{2}\rangle d\mu^{r}(g_{1},\ldots,g_{r})d\mu(h)\\
 & =\int_{(g_{1},\ldots,g_{r})\in G^{r}}\langle\left(\int_{h\in G}\pi(h)\pi(w(g_{1},\ldots,g_{r}))\pi(h)^{-1}d\mu(h)\right)v_{1},v_{2}\rangle d\mu^{r}(g_{1},\ldots,g_{r})\\
 & =\frac{\langle v_{1},v_{2}\rangle}{\dim V}\int_{(g_{1},\ldots,g_{r})\in G^{r}}\trace(\pi(w(g_{1},\ldots,g_{r})))d\mu^{r}(g_{1},\ldots,g_{r})\\
 & =\frac{\langle v_{1},v_{2}\rangle}{\dim V}\tau_{G,\pi}(w),
\end{align*}
where the third equality used Fubini's Theorem and the fourth equality
used Lemma \ref{lem:simple-lemma}. This shows that $\mu_{w}$ is
determined by the values $\tau_{G,\pi}(w)$ with $\pi$ irreducible.
\end{proof}
\begin{proof}[\textbf{Proof of Lemma \ref{lem:convolution-of-words}}]
Suppose for simplicity that $w_{1}$ is generated by $x_{1},\ldots,x_{s}$
and $w_{2}$ is generated by $x_{s+1},\ldots,x_{r}$. Let $(w_{1},w_{2})$
be the map that takes $G^{r}\to G\times G$, $(w_{1},w_{2})(g_{1},\ldots,g_{r})=(w_{1}(g_{1},\ldots,g_{s}),w_{2}(g_{s+1},\ldots g_{r})).$
Let $\nu$ be the pushforward of $\mu^{r}$ under $(w_{1},w_{2})$.
By Fubini's Theorem, the pushforward of a product measure under a
product of two continuous maps is the product of the pushforward measures
of the two maps. Since $\mu^{r}$ is the product measure of $\mu^{s}$
and $\mu^{r-s}$ on $G^{r}=G^{s}\times G^{r-s}$, we obtain $\nu=\mu_{w_{1}}\times\mu_{w_{2}}$.
Furthermore, the word map $w$ is obtained by the composition
\[
G^{r}\xrightarrow{(w_{1},w_{2})}G\times G\xrightarrow{\mathrm{mult}}G
\]
where $\mathrm{mult(g_{1},g_{2})=g_{1}g_{2}}.$ This shows that $\mu_{w}=\mathrm{mult_{*}}[\nu]=\mathrm{mult_{*}}[\mu_{w_{1}}\times\mu_{w_{2}}]=\mu_{w_{1}}*\mu_{w_{2}}$.

If $\mu_{1}$ and $\mu_{2}$ are two conjugation invariant measures
on $G$ and $(\pi,V)$ is an irreducible representation of $G$ then
\begin{align*}
\mu_{1}*\mu_{2}[\trace(\pi)] & =\int_{g_{2}\in G}\int_{g_{1}\in G}\trace(\pi(g_{1}g_{2}))d\mu_{1}(g_{1})d\mu_{2}(g_{2})\\
 & =\int_{h\in G}\int_{g_{2}\in G}\int_{g_{1}\in G}\trace(\pi(hg_{1}h^{-1})\pi(g_{2}))d\mu_{1}(g_{1})d\mu_{2}(g_{2})d\mu(h)\\
 & =\int_{g_{2}\in G}\int_{g_{1}\in G}\trace\left(\left(\int_{h\in G}\pi(h)\pi(g_{1})\pi(h)^{-1}d\mu(h)\right)\pi(g_{2})\right)d\mu_{1}(g_{1})d\mu_{2}(g_{2})\\
 & =\frac{1}{\dim V}\int_{g_{2}\in G}\int_{g_{1}\in G}\trace(\pi(g_{1}))\trace(\pi(g_{2}))d\mu_{1}(g_{1})d\mu_{2}(g_{2})\\
 & =\frac{1}{\dim V}\mu_{1}[\trace(\pi)]\mu_{2}[\trace(\pi)],
\end{align*}
where the second last equality used Lemma \ref{lem:simple-lemma}.
Here we use the notation $\mu[f]$ for the integral of a function
$f$ with respect to a measure $\mu$. The stated formula for $\tau_{G,\pi}(w)$
now follows from $\mu_{w}=\mu_{w_{1}}*\mu_{w_{2}}$ and the fact that
$\mu_{w_{1}}$ and $\mu_{w_{2}}$ are conjugation invariant.
\end{proof}
\begin{proof}[\textbf{Proof of Proposition \ref{prop:char-subgroup-orbits}}]
Let $\Lambda_{1}$ and $\Lambda_{2}$ be the characteristic subgroups
of $\F$ generated by $\Aut(\F).w_{1}$ and $\Aut(\F).w_{2}$ respectively.
Suppose $\Lambda_{1}\neq\Lambda_{2}$. Then at most one of the intersections
\[
\Aut(\F).w_{2}\cap\Lambda_{1},\quad\Aut(\F).w_{1}\cap\Lambda_{2}
\]
is non-empty. Indeed if $\Aut(\F).w_{i}\cap\Lambda_{j}\neq\emptyset$
for $i\neq j$ then since $\Lambda_{j}$ is characteristic, this implies
$\Aut(\F).w_{i}\subset\Lambda_{j}$ and so $\Lambda_{i}\subset\Lambda_{j}$.
So suppose without loss of generality that $\Aut(\F).w_{2}\cap\Lambda_{1}=\emptyset$.
Then (recalling the notation from Example \ref{exa:char-subgroup})
$\tau_{\Lambda_{1}}(w_{2})=0$ but $\tau_{\Lambda_{1}}(w_{1})=1$
showing $w_{1}\stackrel{\mathbf{PosDef}}{\not\sim}w_{2}$.
\end{proof}

\begin{proof}[\textbf{Proof of Lemma \ref{lem:pos-def-on-semidirect-product}}]
Consider a finite sequence of elements $\{(\gamma_{i},\alpha_{i})\}_{i=1}^{N}\subset\Gamma\rtimes\Aut(\Gamma)$.
We need to prove that the matrix $A$ with 
\begin{align*}
A_{ij} & \stackrel{\df}{=}\tau^{+}((\gamma_{i},\alpha_{i})(\gamma_{j},\alpha_{j})^{-1})
\end{align*}
 is positive semidefinite. To this end,
\begin{align*}
A_{ij} & =\tau^{+}((\gamma_{i},\alpha_{i})(\gamma_{j},\alpha_{j})^{-1})\\
 & =\tau^{+}((\gamma_{i},\alpha_{i})(\alpha_{j}^{-1}(\gamma_{j}^{-1}),\alpha_{j}^{-1}))\\
 & =\tau^{+}((\gamma_{i}[\alpha_{i}\alpha_{j}^{-1}](\gamma_{j}^{-1}),\alpha_{i}\alpha_{j}^{-1}))\\
 & =\tau(\gamma_{i}[\alpha_{i}\alpha_{j}^{-1}](\gamma_{j}^{-1}))\\
 & =\tau(\alpha_{i}^{-1}(\gamma_{i})\alpha_{j}^{-1}(\gamma_{j}^{-1}))=\tau(\alpha_{i}^{-1}(\gamma_{i})\alpha_{j}^{-1}(\gamma_{j})^{-1}).
\end{align*}
In other words, $A_{ij}$ is the matrix associated to $\tau$ and
the sequence $\{\alpha_{i}^{-1}(\gamma_{i})\}_{i=1}^{N}$ and so is
positive semidefinite, since $\tau$ is positive definite.
\end{proof}
\begin{proof}[\textbf{Proof of Theorem \ref{thm:extreme-functions-and-ergodic-actions}}]
We use the notation from Example \ref{exa:compact-group-construction-III}.
Suppose first that the action of $\Aut(\F)$ on $G^{r}\cong\Hom(\F,G)$
is not ergodic, so that there exists a Borel set $E\subset G^{r}$
such that $\alpha(E)=E$ for all $\alpha\in\Aut(\F)$ and $0<\mu(E)<1$.
Then letting
\begin{align*}
\tau_{1}(w) & =\frac{1}{\mu(E)\dim V}\int_{{\bf g}\in G^{r}}\trace(\pi(w({\bf g})))\mathbf{1}_{E}(\mathbf{g})d\mu^{r}({\bf g}),\\
\tau_{2}(w) & =\frac{1}{(1-\mu(E))\dim V}\int_{{\bf g}\in G^{r}}\trace(\pi(w({\bf g})))(1-\mathbf{1}_{E}(\mathbf{g}))d\mu^{r}({\bf g}),
\end{align*}
we have that $\tau_{1}$ and $\tau_{2}$ are in $\P_{1}(\F)^{\Aut(\F)}$,
as the measure $\mathbf{1}_{E}(\mathbf{g})d\mu^{r}({\bf g})$ is $\Aut(\F)$-invariant.
On the other hand
\[
\tilde{\tau}_{G,\pi}=\frac{\mu(E)}{2}\tau_{1}+\frac{1-\mu(E)}{2}\tau_{2},
\]
so in this case, $\tilde{\tau}_{G,\pi}$ is not extremal in $\P_{1}(\F)^{\Aut(\F)}$.

Now, for the other direction, suppose that $\pi$ is irreducible and
that the action of $\Aut(\F)$ on $G^{r}$ is ergodic, but for the
sake of a contradiction, suppose that $\tilde{\tau}_{G,\pi}=t\tau_{1}+(1-t)\tau_{2}$
with $t\in(0,1)$ and $\tau_{1},\tau_{2}\in\P_{1}(\F)^{\Aut(\F)}$,
with $\tau_{1}$ not a positive multiple of $\tau_{G,\pi}$. Under
our assumptions we have
\[
\tilde{\tau}_{G,\pi}^{+}=t\tau_{1}^{+}+(1-t)\tau_{2}^{+}
\]
 with $\tau_{1}^{+},\tau_{2}^{+}\in\P_{1}(\F\rtimes\Aut(\F))$, and
$\tau_{1}^{+}$ not a multiple of $\tilde{\tau}_{G,\pi}^{+}$. By
standard facts \cite[Prop. C.5.1]{BdlHV}, this means that $\Pi_{G,\pi}^{+}$
is \uline{reducible} as a unitary representation of $\F\rtimes\Aut(\F)$.
Therefore (see \cite[Proof of Theorem C.5.2]{BdlHV}) there is some
projection $P$ that commutes with all the elements $\Pi_{G,\pi}^{+}(w,\alpha)$,
$P\xi_{G,\pi}^{+}\neq0$, and 
\[
\tau_{3}^{+}(w,\alpha)=\left\langle \Pi_{G,\pi}^{+}(w,\alpha)\frac{P\xi_{G,\pi}^{+}}{\|P\xi_{G,\pi}^{+}\|},\frac{P\xi_{G,\pi}^{+}}{\|P\xi_{G,\pi}^{+}\|}\right\rangle 
\]
is in $\P_{1}(\F\rtimes\Aut(\F))$ with $\tau_{3}^{+}\neq\tilde{\tau}_{G,\pi}^{+}$
(i.e. $\frac{P\xi_{G,\pi}^{+}}{\|P\xi_{G,\pi}^{+}\|}\neq\xi_{G,\pi}^{+}$).
It follows that $P\xi_{G,\pi}^{+}$ is an invariant vector for $\Aut(\F)$
under $\Pi_{0}$. However, when restricted to $\Aut(\F)$, the representation
$\Pi_{0}$ is simply the representation of $\Aut(\F)$ on the $\End(V)$-valued
$L^{2}$ functions on $G^{r}$ acting by permutations of $G^{r}$.
Since $\Aut(\F)$ acts ergodically on $G^{r}$, we must have
\[
\frac{P\xi_{G,\pi}^{+}}{\|P\xi_{G,\pi}^{+}\|}=\int_{G^{r}}^{\oplus}Bd\mu^{r}(\mathbf{g}),
\]
where $B\in\End(V)$ is a constant with $\trace(BB^{*})=1$. But this
means in turn, using the invariance of Haar measure under conjugation,
\begin{align*}
\tau_{3}^{+}(w,\alpha) & =\tau_{3}^{+}(w,e)=\int_{G^{r}}\trace(\pi(w(\mathbf{g}))BB^{*})d\mu^{r}(\mathbf{g})\\
 & =\int_{(g_{1},\ldots,g_{r})\in G^{r}}\left(\int_{h\in G}\trace\left(\pi(w(hg_{1}h^{-1},\ldots,hg_{r}h^{-1}))BB^{*}\right)d\mu(h)\right)d\mu^{r}(g_{1},\ldots,g_{r})\\
 & =\int_{(g_{1},\ldots,g_{r})\in G^{r}}\left(\int_{h\in G}\trace\left(\pi(h)\pi(w(g_{1},\ldots,g_{r}))\pi(h)^{-1}BB^{*}\right)d\mu(h)\right)d\mu^{r}(g_{1},\ldots,g_{r})\\
 & =\frac{1}{\dim V}\int_{(g_{1},\ldots,g_{r})\in G^{r}}\trace(\pi(w({\bf g})))\trace(BB^{*})d\mu^{r}(\mathbf{g})\\
 & =\frac{1}{\dim V}\int_{(g_{1},\ldots,g_{r})\in G^{r}}\trace(\pi(w({\bf g})))d\mu^{r}(\mathbf{g})=\tilde{\tau}_{G,\pi}^{+}(w,\alpha).
\end{align*}
The second last equality used Lemma \ref{lem:simple-lemma}. This
is a contradiction.
\end{proof}
\begin{proof}[\textbf{Proof of Theorem \ref{thm:hierarchy-collapse}}]
If $\underline{x}\in\GL_{r}(\Z).\underline{y}$ then it is easy to
check that $\mu_{\underline{x}}=\mu_{\underline{y}}$ on any finite
abelian group.

The other direction is the more interesting one. Assume that $\underline{x}\notin\GL_{r}(\Z).\underline{y}$.
The orbit of $\underline{x}=(x_{1},\ldots,x_{r})$ is parametrized
by the modulus of the greatest common divisor of the $x_{i}$ (which
we take to be $\infty$ if $\underline{x}=\underline{0})$, and similarly
for $\underline{y}$.

Thus our assumptions entail, by switching $\underline{x}$ and $\underline{y}$
if necessary, that there is a prime $p$ and an exponent $f$ such
that $\underline{x}\equiv0\bmod p^{f}$ and $\underline{y}\not\equiv0\bmod p^{f}$.
This means that for any $\mathbf{g}=(g_{1},\ldots,g_{r})\in(\Z/p^{f}\Z)^{r}$,
$x_{1}g_{1}+\cdots+x_{r}g_{r}=\underline{0}$, so the $\underline{x}$-measure
on $\Z/p^{f}\Z$ is an atom at $\underline{0}$. On the other hand,
$\underline{y}\not\equiv0\bmod p^{f}$ implies there is some $\mathbf{g}=(g_{1},\ldots,g_{r})\in(\Z/p^{f}\Z)^{r}$
such that $y_{1}g_{1}+\cdots+y_{r}g_{r}\neq\underline{0},$ so the
$\underline{y}$-measure on $\Z/p^{f}\Z$ is not supported at $\underline{0}\in\Z/p^{f}\Z$.
This proves the $\underline{x}$- and $\underline{y}$-measures on
$\Z/p^{f}\Z$ are distinct. 

\end{proof}

\bibliographystyle{alpha}

\newpage{}

\begin{multicols}{2}

\noindent Benoît Collins, \\
Department of Mathematics, 

\noindent Kyoto University, 

\noindent Kyoto 606-8502, 

\noindent Japan 

\noindent \texttt{collins@math.kyoto-u.ac.jp}\\
\\
Michael Magee, \\
Department of Mathematical Sciences,\\
Durham University, \\
Lower Mountjoy, DH1 3LE Durham,\\
United Kingdom

\noindent \texttt{michael.r.magee@durham.ac.uk}\\

\noindent \columnbreak

\noindent Doron Puder, \\
School of Mathematical Sciences, \\
Tel Aviv University, \\
Tel Aviv, 6997801, Israel\\
\texttt{doronpuder@gmail.com}

\end{multicols}
\end{document}